\documentclass[11pt,reqno]{amsart}
\usepackage{amsthm, amsmath, amssymb, amscd, latexsym, multicol, verbatim, enumerate, }
\usepackage[usenames]{color}
\usepackage{hyperref}
\usepackage{float}
\usepackage{tikz}
\usepackage[all]{xy}

\newtheorem{theorem}{Theorem}[section]
\newtheorem*{theorem*}{Theorem}
\newtheorem{lemma}[theorem]{Lemma}
\newtheorem*{lemma*}{Lemma}
\newtheorem{corollary}[theorem]{Corollary}
\newtheorem*{corollary*}{Corollary}
\newtheorem{proposition}[theorem]{Proposition}
\newtheorem*{proposition*}{Proposition}
\newtheorem{remark}[theorem]{Remark}
\newtheorem{definition}[theorem]{Definition}

\newtheorem{example}[theorem]{Example}

\newcommand{\ncom}{\newcommand}
\ncom{\rar}{\rightarrow}
\ncom{\lrar}{\longrightarrow}
\ncom{\ov}{\overline}
\ncom{\m}{\mbox}
\ncom{\sta}{\stackrel}

\ncom{\comx}{{\mathbb C}}
\ncom{\Z}{{\mathbb Z}}\ncom{\Q}{{\mathbb Q}}\ncom{\R}{{\mathbb R}}\ncom{\G}{{\mathbb G}}\ncom{\al}{\alpha} \ncom{\p}{{\mathbb P}}\ncom{\E}{{\mathbb E}}\ncom{\N}{{\mathbb N}}\ncom{\K}{\mathbb K} \ncom{\Le}{{\mathbb L}}\ncom{\A}{{\mathbb A}}\ncom{\B}{{\mathbb B}}\ncom{\F}{{\mathbb F}}\ncom{\C}{{\mathbb C}}\ncom{\f}{\frac}\ncom{\cA}{{\mathcal A}}\ncom{\cX}{{\mathcal X}}\ncom{\cO}{{\mathcal O}}\ncom{\cW}{{\mathcal W}}\ncom{\cL}{{\mathcal L}}\ncom{\cP}{{\mathcal P}}\ncom{\cH}{{\mathcal H}}\ncom{\cS}{{\mathcal S}}\ncom{\cM}{{\mathcal M}}\ncom{\cC}{\mathcal C}\ncom{\cT}{{\mathcal T}}\ncom{\cF}{{\mathcal F}}\ncom{\cN}{{\mathcal N}}\ncom{\cJ}{{\mathcal J}}\ncom{\cV}{{\mathcal V}}\ncom{\cZ}{{\mathcal Z}}\ncom{\cU}{{\mathcal U}}\ncom\cSU{{\mathcal S \mathcal U}}\ncom{\cG}{{\mathcal G}}\ncom{\cQ}{{\mathcal Q}}\ncom{\cR}{{\mathcal R}}\ncom{\cE}{{\mathcal E}}\ncom{\cY}{{\mathcal Y}}\ncom{\cD}{{\mathcal D}}
\ncom\bs{\mathbf{s}} \ncom\bm{\mathbf{m}} \ncom\bt{\mathbf{t}} \ncom\bp{\mathbf{p}} \ncom\bq{\mathbf{q}}

\ncom\lie{\mathfrak} \ncom{\bc}{\mathbb C} \ncom\om{\omega}

\title[Minuscule Schubert varieties, Kogan faces and poset polytopes]{Minuscule Schubert varieties: Poset polytopes, PBW-degenerated Demazure modules, and Kogan faces}
\author{Rekha Biswal and Ghislain Fourier}
\address{\newline The Institute of Mathematical Sciences, CIT Campus, Taramani, Chennai, India} 
\email{rekha@imsc.res.in}
\address{\newline
Mathematisches Institut, Universit\"at Bonn, Germany\newline
School of Mathematics and Statistics, University of Glasgow, UK}
\email{ghislain.fourier@glasgow.ac.uk}
\begin{document}
\begin{abstract}
We study a family of posets and the associated chain and order polytopes. We identify the order polytope as a maximal Kogan face in a Gelfand-Tsetlin polytope of a multiple of a fundamental weight. We show that the character of such a Kogan face equals to the character of a Demazure module which occurs in the irreducible representation of $\mathfrak{sl}_{n+1}$ having highest weight multiple of fundamental weight and for any such Demazure module there exists a corresponding poset and associated maximal Kogan face.\\
We prove that the chain polytope parametrizes a monomial basis of the associated PBW-graded Demazure module and further, that the Demazure module is a favourable module, e.g. interesting geometric properties are governed by combinatorics of convex polytopes. Thus, we obtain for any minuscule Schubert variety a flat degeneration into a toric projective variety which is projectively normal and arithmetically Cohen-Macaulay.\\
We provide a necessary and sufficient condition on the Weyl group element such that the toric variety associated to the chain polytope and the toric variety associated to the order polytope are isomorphic.
\end{abstract}
\maketitle

\section*{Introduction}
Toric degenerations of Schubert varieties have been studied quite a lot in the last decades. The most famous might be the degeneration via Gelfand-Tsetlin polytopes, see for example \cite{GL96}. There are various other degenerations, for example for every reduced decomposition of the longest Weyl group element \cite{Lit98, AB04}. Here is our approach to degenerations of minuscule Schubert varieties, e.g. Schubert varieties of a fixed Grassmannian $\operatorname{Gr } (i,n+1)$.\\
Consider the lattice points in the rectangle of height $n-i$ and width $i-1$, where the left column is labeled by $1$, the right column by $i$, the bottom row by $n$ and the top row by $i$. We consider the partial order induced by the cover relation $(k,j) \geq (k+1,j), (k,j-1)$, so $(k,j)$ is bigger than any of its right and upper neighbor. For a fixed sequence of integers $\underline{\ell} = i-1 \leq \ell_1 \leq \ldots \leq \ell_i \leq n$ we consider the subposet $\mathcal{P}_{\underline{\ell}}$ of lattice points 
in the right upper corner i.e union of all $(k,j)$ with $i \leq j \leq \ell_k$. AS we will see, all these $\underline{\ell}$ parametrize naturally the Schubert varieties in $\operatorname{Gr } (i,n+1)$.\\
Stanley (\cite{Sta86, Lit98}) has associated to each finite poset two polytopes, the order polytope and the chain polytope. Let $N = |\mathcal{P}_{\underline{\ell}}| = \sum (\ell_k +1 - i)$, then these polytopes are in our case
\[
\mathcal{O}_{\underline{\ell}} :=  \{ (x) \in \mathbb{R}^{N} \mid x_a \leq x_b \text{ if } a < b \, , \, 0 \leq x_a \leq 1 \; \; \forall \, a,b \in \mathcal{P}_{\underline{\ell}} \} \\
\]
\[
\mathcal{C}_{\underline{\ell}} := \{ (x) \in \mathbb{R}^{N} \mid 0 \leq x_{a_1} + \ldots + x_{a_s} \leq 1 \; \; \forall \, \text{ chains } a_1 < \ldots < a_s \in \mathcal{P}_{\underline{\ell}} \}.
\]
The two polytopes are related as follows:
\begin{enumerate}
\item The Ehrhart polynomial is the same for both polytopes (\cite{Sta86}).
\item $\mathcal{C}_{\underline{\ell}}$ and $\mathcal{O}_{\underline{\ell}}$ are normal (Corollary~\ref{cor-nor}).
\item $\mathcal{C}_{\underline{\ell}}$ is Gorenstein $\Leftrightarrow \mathcal{O}_{\underline{\ell}}$ is Gorenstein $\Leftrightarrow \mathcal{P}_{\underline{\ell}}$ is a pure poset (Corollary~\ref{cor-gor} and \cite{Sta78}). 
\item $\mathcal{C}_{\underline{\ell}}$ is unimodular equivalent to $\mathcal{O}_{\underline{\ell}}$ if and only if  $\ell_{i-1} < i+2
$ or  $\ell_{i-2} < i+1$. (\cite{HL12} and Lemma~\ref{lem-compare}).
\end{enumerate}

Let $\lie{sl}_{n+1} = \lie n^- \oplus \lie h \oplus \lie n^+$ . For $i$-th fundamental weight $\omega_i$ of $\lie{sl}_{n+1}$, let $V(m \omega_i)$ be the simple module of $\lie{sl}_{n+1}$ with highest weight $m \omega_i$ and $v_m$ a highest weight vector. For any $w \in W=S_{n+1}$, the Weyl group, the extremal weight space in $V( m \omega_i)$ of weight $w(m \omega_i)$ is one-dimensional.\\ 
We denote the $\lie n^+ \oplus \lie h$-submodule generated by this extremal weight space as $V_w(m \omega_i)$, called the \textit{Demazure module}. Certainly $V_{w \tau}(m \omega_i) = V_w(m \omega_i)$ for any $\tau \in S_i \times S_{n+1-i}$ and hence we can always assume that $w$ is of minimal length among all representatives of its class in $W^i:= S_{n+1}/S_{i} \times S_{n+1-i}$.\\
Let $\lie n_w^- = w^{-1} (\lie n^+) w \cap \lie n^-$ and denote $-R^-_w$ the set of roots in $\lie n^-_w$. $R^-_w$ is equipped with a natural partial order, $\alpha \geq \beta$, if their difference is again a positive root and then $R^-_w$ is isomorphic to some $\mathcal{P}_{\underline{\ell}_w}$ with $\operatorname{length }(w) = N$ and vice versa, there exists such a $w$ for all sequences $\underline{\ell}$ (Proposition~\ref{prop-pos}). Via conjugation with $w^{-1}$ we can identify $V_w(m \omega_i)$ with $U(\lie n^-_w).v_m$. 

We recall the PBW filtration on the universal enveloping algebra and the induced filtration on  $M = w^{-1} V_w(m \omega_i) w$ (this can be found more general in \cite{FFoL13a}):
\[
U(\lie n^-_w)_s := \langle x_{i_1} \cdots x_{i_r} \, | \, r \leq s, x_{i_j} \in \lie n^-_w \rangle \; ; \; M_s := U(\lie n^-_w)_s.v_m.
\]
The associated graded module is denoted $M^a$ and $S(\lie n^-_w)$ acts cyclically on this module. This implies that there exists $I \subset S(\lie n^-_w)$ with $M^a \cong S(\lie n^-_w)/I$. For $\bs \in \mathbb{Z}_{\geq 0}^{N}$ we associate $f^{\bs} = \prod_{\alpha \in R^-_w} f_\alpha^{s_\alpha} \in S(\lie n^-_w)$. The first main theorem of this paper is:
\begin{theorem*} Let $w \in W^i$ and $m \geq 0$.
\begin{enumerate}
\item $\{f^\bs.v_m \, | \, \bs \in m\mathcal{C}_{\underline{\ell}_w} \}$ is a basis of $(U(\lie n^-_w).v_{m})^a$.
\item $I$ is generated by $\left(U(\lie n^+).\{ f_\alpha^{m+1} \, | \, \alpha \in R^-_w \} \right) \cap S(\lie n^-_w)$.
\item $\{w(f^\bs).v_{w(m \omega_i)} \, | \, \bs \in m\mathcal{C}_{\underline{\ell}_w} \}$ is a basis of $V^a_w(m \omega_i)$.
\item Suppose $w \leq \tau \in W^i$, then $m\mathcal{C}_{\underline{\ell}_w}$ is a face of $m\mathcal{C}_{\underline{\ell}_\tau}$.
\item Let $w_0$ be the minimal representative of the longest element in $W^i$, then $m\mathcal{C}_{\underline{\ell}_{w_0}} = P(m \omega_i)$, the polytope defined in \cite{BD14}. 
\end{enumerate}
\end{theorem*}
\begin{example}
Let us consider here $\operatorname{Gr}(2,4)$, then we have six Weyl group elements. We focus here on $w = s_1s_3s_2$, then $\underline{\ell} = (2 \leq 3)$, the chain polytope is described as $0 \leq (2,2) \leq (1,2) \leq 1$ and $0 \leq (2,2) \leq (2,3) \leq 1$.  Then the basis of $V^a_w( \omega_2)$ is given by 
\[
\{ v_2\wedge v_4 \, , \, e_{\alpha_3}.v_2\wedge v_4 \, , \, e_{\alpha_1}.v_2\wedge v_4 \, , \,  e_{\alpha_3}e_{\alpha_1}.v_2\wedge v_4 \, , \, e_{\alpha_1 + \alpha_2 + \alpha_3}.v_2\wedge v_4\}.
\]
In this case, the polytope is unimodular equivalent to the face of the Gelfand-Tsetlin polytope $GT(\omega_2)$ defined by $x_{1,1} = x_{2,1}$ (for notations see Section~\ref{Section4}.
\end{example}
This theorem is analogous to the previous results for simple modules ($w$ being the longest Weyl group element) in the $A_n$-type case (\cite{FFoL11a}), 
$C_n$-type case (\cite{FFoL11b}), in both case over the integers (\cite{FFoL13}), for cominuscule weights  (\cite{BD14}), for Demazure modules for triangular Weyl group elements 
(including Kempf elements) in the $A_n$-type case (\cite{Fou14b}), partial results are also known for fusion products of simple modules (\cite{Fou14}). There is a particular interest in the graded character of these modules due to the conjectured connection to 
Macdonald polynomials (\cite{CO13, CF13, FM14, BBDF14}).\\

The famous Gelfand-Tsetlin polytope $GT(\lambda)$ parametrizes a monomial basis of $V(\lambda)$ \cite{GC50}. In \cite{Kog00} certain faces of this polytope were studied and the Weyl group type of such a face was introduced. Then \cite{KST12} showed that the Demazure character of $V_w(\lambda^*)$ for a regular $\lambda$ is equal to the character of the union of all faces of type $ww_0$ (here $w_0$ denotes the longest element in $W$).\\
We adapt the notion of Kogan faces to the special case $\lambda = m \omega_{n+1-i}$ and $w_0$ the minimal representative of the longest element in $W^i$, then
\begin{theorem*} Let $w \in W^i$ and $\tau = ww_0^{-1}$  and $m \geq 0$.
\begin{enumerate}
\item There is a unique maximal Kogan face $F_{\tau}(m\omega_{n+1-i})$ of type $\tau$  in $GT(m \omega_{n+1-i})$.
\item The character of $F_{\tau}(m\omega_{n+1-i})$ is equal to the character of $V_w(m\omega_i)$.
\item The defining poset structure on $F_{\tau}(\omega_{n+1-i})$ is isomorphic to $\mathcal{P}_{\underline{\ell}_w}$ and so \\$F_{\tau}(m\omega_{n+1-i}) = m\mathcal{O}_{\underline{\ell}_w}$.
\end{enumerate}
\end{theorem*}

Let $\mathbb{U}$ be the group of invertible strictly upper triangular matrices (with Lie algebra $\lie n^+$) and $\mathbb{U}^a$ be the algebraic group of $\dim \mathbb{U}$-copies of the additive group. \\
Let $\prec$ be a total order on the set of positive roots, which extends to a homogeneous lexicographic order on $U(\lie n^+)$. 
The associated graded space $V_w(m \omega_i)^t$ of the induced filtration on $V_w(m \omega_i)$ is then a $\mathbb{U}^a$-module such that all graded components are at 
most one-dimensional and there is a unique monomial basis of this space. We call the monomials of this basis essential. $V_w(m \omega_i)$ is 
called \textit{favourable} if there exists a normal polytope whose lattice points are exactly the essential monomials and the lattice points in the $n$-th 
dilation parametrize a basis of the Cartan component of the $n$-times tensor product of $V_w(m \omega_i)$ (\cite{FFoL13a}). We can show:
\begin{proposition*}
$V_w(m \omega_i)$ is a favourable module. 
\end{proposition*}

$V_w(m \omega_i)$ is a $\mathbb{U}$-module, while $V_w(m \omega_i)^a$ and $V_w(m \omega_i)^t$ are $\mathbb{U}^a$-modules. We set
\[
X_w^a := \overline{\mathbb{U}^a.[m]} \subset \mathbb{P}(V_w(m\omega_i)^a)  \; ; \; X_w^t := \overline{\mathbb{U}^a.[m]} \subset \mathbb{P}(V_w(m\omega_i)^t).
\]

It has been shown in \cite{FFoL13a} that the property \textit{favourable} imply certain interesting geometric properties of the corresponding varieties, so the proposition implies 
\begin{corollary*}
$X_w^t$ is a flat toric degeneration of $X_w^a$ and both are flat degenerations of $X_w$ into projectively normal and arithmetically Cohen-Macaulay varieties. 
\end{corollary*}

On the other hand, as shown in \cite{GL96}, the Gelfand-Tsetlin polytope induces a toric degeneration $X_w^{gt}$ of the Schubert variety. Using our results on the order and chain polytope we deduce:
\begin{corollary*} Let $w \in W^i$. Then $X_w^t$ and $X_w^{gt}$ are isomorphic as toric varieties if and only if there is no reduced decomposition of $w$ of the form $w = \ldots(s_{i-1} s_{i-2})(s_{i+1}s_i s_{i-1})(s_{i+2}s_{i+1}s_i)$.
\end{corollary*}
We see immediately that the toric varieties are isomorphic for all Weyl group elements if $\lambda = m \omega_i$ with $i \in \{1,2,n-1,n\}$ and so especially for all Weyl group elements for $\lie{sl}_n$ with $n \leq 5$.\\
It would be interesting to compare other toric degenerations of Schubert varieties to $X_w^t$, as there is for any reduced decomposition of the longest Weyl group element an induced toric degeneration. Note for example \cite[Example 5.9]{AB04} which is similar to the above corollary, namely the degenerated Schubert variety for fundamental weights does not depend on the chosen reduced decomposition for $\lie{sl}_n$ with $n \leq 5$. \\


\textbf{Acknowledgments} G.F. is funded by the DFG priority program ''Representation theory''. Both authors would like to thank the Centre de recherches math\'ematiques in Montreal for the hospitality during the thematic semester ''New direction in Lie theory'', where this cooperation has been started.


\section{Posets and polytopes}\label{Section1}
We introduce a poset $\mathcal{P}_{\underline{\ell}}$ and consider two polytopes associated to this poset. Let $1 \leq i \leq n$ be two fixed integer and $\underline{\ell}$ an ascending sequence of integers
\[\underline{\ell} := (i-1 \leq \ell_1 \leq \ldots \leq \ell_{i} \leq n) \; \; , \; \; \text{we set }\ell_0 := i-1.\]
\noindent Let $\mathcal{P}_{\underline{\ell}}$ be the poset with vertices
\[ \{ x_{k,j} \, | \,  1 \leq k \leq i \, , \,  i \leq j \leq \ell_k \} \]
and relations:
\[
x_{k_1, j_1} \geq x_{k_2, j_2} :\Leftrightarrow k_1 \leq k_2 \text{ and } j_1 \geq j_2.
\]
Then $\mathcal{P}_{\underline{\ell}}$ has a unique minimal element $x_{i,i}$ and several maximal elements $\{ x_{k, \ell_k} \, | \, \ell_k \neq \ell_{k-1} \}$.
The number of vertices in $\mathcal{P}_{\underline{\ell}}$ is 
\[
N := \sum_{k=1}^{i} \ell_k -i +1.
\]
\begin{example}\label{exam-1}
Let $i = 4, n=6$, $\underline{\ell} = (4,5,6,6)$, then $\mathcal{P}_{\underline{\ell}}$:\\
$$
\begin{array}{ccccccc}
x_{1,4} & \rightarrow & x_{2,4} & \rightarrow & x_{3,4} & \rightarrow & x_{4,4}\\
  && \uparrow & &\uparrow && \uparrow\\
	&& x_{2,5} & \rightarrow &x_{3,5} & \rightarrow  & x_{4,5}\\
	&& && \uparrow && \uparrow\\
	&& && x_{3,6} & \rightarrow &x_{4,6} \\
\end{array}
$$
(here an arrow $x \longrightarrow y$ iff $x$ covers $y$). The maximal elements are $x_{1,4}, x_{2,5}, x_{3,6}$. 
\end{example}


\subsection{Chain and order posets}
Stanley (\cite{Sta86}) has associated two polytopes with each finite poset. We recall these two polytopes in our context. The first one, which is well known and studied is the order polytope
\[
\mathcal{O}_{\underline{\ell}} : = \{ (s_{k,j}) \in (\mathbb{R}_{\geq 0})^{N} \, | \, s_{k,j} \leq 1 \text{ for all } k,j \, , \, s_{k_1, j_1} \geq s_{k_2, j_2} \text{  if } x_{k_1, j_1} > x_{k_2, j_2},  \}
\]
We see straight that this is a $[0,1]$-polytope and hence bounded.\\
The second polytope is the chain polytope
\[
\mathcal{C}_{\underline{\ell}} : = \{ (s_{k,j}) \in (\mathbb{R}_{\geq 0})^{N} \, | \, \sum_{p} s_{k_p, j_p} \leq 1 \text{ for all chains }  x_{k_1, j_1} > \ldots > x_{k_s, j_s}  \}
\]
Of course it is enough to consider only maximal chains. Again, this polytope is bounded as every $x_{k,j}$ is in a maximal chain and hence bounded.\\

\begin{remark}
We should remark here that this construction has been generalized to marked posets by \cite{ABS11}, also related to PBW filtrations. 
But for our purpose the polytopes defined by Stanley will be enough. In fact, there is still a strong connection between marked chain polytopes and 
PBW degenerated Demazure modules on one hand and  marked order polytopes and Gelfand-Tsetlin polytopes on the other hand, as for example shown in \cite{ABS11} 
for the longest Weyl group element. The connection between certain PBW-graded Demazure modules and the marked chain polytope has been investigated for certain 
Weyl group elements (\cite{Fou14b}). We expect this connection for all Weyl group elements, this is part of ongoing research.
\end{remark}

The order and the chain polytope share several interesting properties. For any $t \in \mathbb{Z}_{\geq 1}$ we denote 
\[
\mathcal{O}_{\underline{\ell}}(t) := | t \mathcal{O}_{\underline{\ell}} \cap (\mathbb{Z}_{\geq 0})^N|  \; , \; \mathcal{C}_{\underline{\ell}}(t) := | t \mathcal{C}_{\underline{\ell}} \cap (\mathbb{Z}_{\geq 0})^N|.
\]
Then it is a classical result due to Ehrhart that both functions are actually polynomials, called the Ehrhart polynomials. In fact, this result by Ehrhart is valid for all lattice polytopes. 
\begin{theorem*}[\cite{Sta86}]\label{sta-1}
For all $\underline{\ell}$, the two polynomials coincide: $\mathcal{O}_{\underline{\ell}}(t) = \mathcal{C}_{\underline{\ell}}(t) $.
\end{theorem*}
This implies that the $t$-dilation of the polytopes $\mathcal{O}_{\underline{\ell}}, \mathcal{C}_{\underline{\ell}} $ do have the same number of lattice points. 


\subsection{Faces of the polytopes}
Another result due to Stanley is concerned with the number of vertices, or $0$-dimensional faces, in these polytopes.
\begin{lemma*}[\cite{Sta86}]
The number of vertices in $\mathcal{O}_{\underline{\ell}}$ and $\mathcal{C}_{\underline{\ell}}$ coincide.
\end{lemma*}
While the higher dimensional faces are way more complicated to understand, we can at least give formulas for the number of facets, the $n-1$-dimensional faces.
\begin{lemma}
Let $\underline{\ell} = (i-1 \leq \ell_1 \leq \ldots \leq \ell_{i} \leq n)$, then the number of facets of $\mathcal{O}_{\underline{\ell}}$ is:
\[
 1 + \sharp \{ k \, | \, \ell_k \neq \ell_{k+1} \} + 1 + \sum_{\ell_k \geq i} (\ell_k - i) + \sum_{ k < i} (\ell_k +1 - i) 
\]
and the number of facets of $\mathcal{C}_{\underline{\ell}}$ is:
\[
N + \sum_{\ell_k \neq \ell_{k-1}} \binom{\ell_k - k} {i - k}
\]
(here we set $\ell_{0} = i-1$).
\end{lemma}
\begin{proof}
In \cite{Sta86}, it is shown that the number of facets in the order polytope is equal to the number of cover relations plus the number of minimal elements plus the number of maximal elements. Further, the number of facets in the chain polytope is equal to the number of vertices (of the poset) plus the number of maximal chains in the poset. The easy computation of these numbers proves the lemma.
\end{proof}
\begin{example}
We turn again to Example~\ref{exam-1}. $i=4, n=6$, $\underline{\ell} = (4,5,6,6)$. Then the number of facets in the order polytope $\mathcal{O}_{\underline{\ell}}$ is
\[
1 + 2 + 1 + (4-4) + (5-4) + (6-4) + (6-4) + (5-4) + (6-4) + (7-4) = 15, 
\] 
while the number of facets in the chain polytope $\mathcal{C}_{\underline{\ell}}$ is
\[
9 + \binom{3}{3} + \binom{5-2}{4-2} + \binom{6-3}{4-3} = 9 + 1 + 3 + 3 = 16.
\]
\end{example}
We see that these numbers are not equal in general, in fact in \cite{HL12} the following equivalences are proven
\begin{theorem*}\label{thm-equiv} The following are equivalent
\begin{enumerate}
\item The number of facets in  $\mathcal{O}_{\underline{\ell}}$ and $\mathcal{C}_{\underline{\ell}}$ is the same.
\item The poset $\mathcal{P}_{\underline{\ell}}$ does not contain a subposet $\{x_1, x_2, x_3, x_4, x_5\}$ such that $x_1, x_2 > x_3 > x_4, x_5$.
\item The polytopes $\mathcal{O}_{\underline{\ell}}$ and $\mathcal{C}_{\underline{\ell}}$ have the same $\mathbf{f}$-vector.
\item The polytopes $\mathcal{O}_{\underline{\ell}}$ and $\mathcal{C}_{\underline{\ell}}$ are unimodular equivalent.
\end{enumerate}
\end{theorem*}
\begin{remark} The condition that the poset does not contain a subposet $\{x_1, x_2, x_3, x_4, x_5\}$ such that $x_1, x_2 > x_3 > x_4, x_5$ translates in our case to: $\ell_{i-2} < i+1$ or $\ell_{i-1} < i+2$.
\end{remark}


\subsection{Dilation and normality}
Here we consider dilations of our polytopes, for any $m \in \mathbb{Z}_{\geq 1}$:
\[
m\mathcal{O}_{\underline{\ell}}  = \{ (s_{k,j}) \in (\mathbb{R}_{\geq 0})^{N} \, | \, s_{k_1, j_1} \geq s_{k_2, j_2} \text{  if } x_{k_1, j_1} > x_{k_2, j_2} \, , \, s_{k,j} \leq m \, \forall \, k,j\}
\]
\[
m\mathcal{C}_{\underline{\ell}}  = \{ (s_{k,j}) \in (\mathbb{R}_{\geq 0})^{N} \, | \, \sum_{p} s_{k_p, j_p} \leq m \text{ for all chains }  x_{k_1, j_1} > \ldots > x_{k_s, j_s}  \}.
\]

\begin{lemma}\label{lem-mink}
For any $m \in \mathbb{Z}_{\geq 1}$, then we have for the Minkowski sum of the lattice points:
\[
(m+1)\mathcal{C}_{\underline{\ell}} \cap \mathbb{Z}_{\geq 0}^N =  \left(m\mathcal{C}_{\underline{\ell}}  \cap \mathbb{Z}_{\geq 0}^N \right) +  \left(\mathcal{C}_{\underline{\ell}}  \cap \mathbb{Z}_{\geq 0}^N\right)
\]
\end{lemma}
\begin{proof}
Let $\bs \in (m+1)\mathcal{C}_{\underline{\ell}}  \cap \mathbb{Z}_{\geq 0}^N$, and denote 
\[
\operatorname{supp }(\bs)=\{ x_{k,j}: s_{k,j}\neq 0\}.
\]
Then $\operatorname{supp }(\bs)$ is again a poset. If $M$ denotes the subset of minimal elements in $\operatorname{supp }(\bs)$, then the elements of $M$ are pairwise incomparable. We define $\bt \in \mathbb{Z}_{\geq 0}^N $ with $t_{k,j} = 1$ if $x_{k,j} \in M$, $t_{k,j} = 0$ else.\\
Since the elements of $M$ are pairwise not comparable, every chain in $\mathcal{P}_{\underline{\ell}}$ has at most one element from $M$. So $ \bt \in \mathcal{C}_{\underline{\ell}}  \cap \mathbb{Z}_{\geq 0}^N$. To prove the theorem, it is enough to show that $\bs - \bt \in m\mathcal{C}_{\underline{\ell}}  \cap \mathbb{Z}_{\geq 0}^N$.\\
Let $P$ be a maximal chain in $\mathcal{P}_{\underline{\ell}}$. Then 
\[
\sum_{x_{k,j} \in P} s_{k,j} \leq m+1,
\] 
and so if $P \cap M \neq \emptyset$ then 
\begin{eqnarray}\label{eq-mink}
\sum_{x_{k,j} \in P}  s_{k,j} - t_{k,j} \leq m.
\end{eqnarray}
Suppose now $P \cap M = \emptyset$. Let $x_{k,j}$ be the minimal element in $P$ with $s_{k,j} \neq 0$, then by construction of $M$, there exists $x_{k', j'} \in M$ with $x_{k', j'} < x_{k,j}$ and $s_{k', j'} \neq 0$. So we can construct a new chain $P'$ consisting of $\{ x_{p,q} \in P \, | \, x_{p,q} \geq x_{k,j} \} \cup \{x_{k',j'} \}$. Then 
\[
\sum_{x_{k,j} \in P} s_{k,j} < \sum_{x_{k,j} \in P'} s_{k,j} 
\]
and since  $P' \cap M \neq  \emptyset$ we have with \eqref{eq-mink}
\[
\sum_{x_{k,j} \in P}  s_{k,j} -  t_{k,j} \leq \sum_{x_{k,j} \in P'}  s_{k,j}  - t_{k,j} \leq m. 
\]
This implies that $\bs - \bt \in m\mathcal{C}_{\underline{\ell}}  \cap \mathbb{Z}_{\geq 0}^N$.
\end{proof}

The following should be well-known from the literature, we include the proof for the readers convenience:
\begin{lemma}\label{lem-mink2}
For any $m \in \mathbb{Z}_{\geq 1}$:
\[
(m+1)\mathcal{O}_{\underline{\ell}} \cap \mathbb{Z}_{\geq 0}^N =  \left(m\mathcal{O}_{\underline{\ell}}  \cap \mathbb{Z}_{\geq 0}^N\right) + \left( \mathcal{O}_{\underline{\ell}} \cap \mathbb{Z}_{\geq 0}^N \right)
\]
\end{lemma}
\begin{proof}
Let $\bs \in (m+1)\mathcal{O}_{\underline{\ell}} \cap \mathbb{Z}_{\geq 0}^N$, then we define $\bt \in \mathbb{Z}_{\geq 0}^N$ as
\[
t_{k,j} := \begin{cases} 1 & \text{ if }  s_{k,j} \neq 0 \\ 0 & \text{ if } s_{k,j} = 0 \end{cases}
\]
Since $\bs \in(m+1)\mathcal{O}_{\underline{\ell}} \cap \mathbb{Z}_{\geq 0}^N$, we have $s_{k,j} \neq 0 \Rightarrow s_{k',j'} \neq 0$ for all $x_{k', j'} > x_{k,j}$. This implies that $\bt \in \mathcal{O}_{\underline{\ell}} \cap \mathbb{Z}_{\geq 0}^N$. It is furthermore obvious that
\[
\bs - \bt \in m\mathcal{O}_{\underline{\ell}} \cap \mathbb{Z}_{\geq 0}^N. 
\] 
\end{proof}
\begin{definition}
A lattice polytope $P$ is called \textit{normal} if 
\[
\forall \, m \geq 1 \text{ and } \forall \; p \in mP \cap \mathbb{Z}^{\operatorname{dim } P} \, : \, \exists \; p_1, \ldots, p_m \in P \cap \mathbb{Z}^{\operatorname{dim } P} \, \text{ with }  \, p = p_1 + \ldots  + p_m.
\]
\end{definition} 
So every lattice point in the $m$-th dilation of $P$ can be written as a sum of $m$ lattice points in $P$.
We can deduce from Lemma~\ref{lem-mink2} and Lemma~\ref{lem-mink} easily for all $n,m \geq 1$:
\[
(m+n)\mathcal{O}_{\underline{\ell}} \cap \mathbb{Z}_{\geq 0}^N =  \left(m\mathcal{O}_{\underline{\ell}} \cap \mathbb{Z}_{\geq 0}^N \right) + \left( n\mathcal{O}_{\underline{\ell}} \cap \mathbb{Z}_{\geq 0}^N \right)
\]
and
\[
(m+n)\mathcal{C}_{\underline{\ell}} \cap \mathbb{Z}_{\geq 0}^N =  \left(m\mathcal{C}_{\underline{\ell}} \cap \mathbb{Z}_{\geq 0}^N\right) + \left( n\mathcal{C}_{\underline{\ell}}  \cap \mathbb{Z}_{\geq 0}^N\right).
\]
\begin{corollary}\label{cor-nor}
For fixed $1 \leq i \leq n$ and $\underline{\ell}\,:$ The polytopes $\mathcal{O}_{\underline{\ell}}$ and $\mathcal{C}_{\underline{\ell}}$ are both normal.
\end{corollary}


\section{Some combinatorics on roots}\label{Section2}
\subsection{Preliminaries}
We consider $\lie {sl}_{n+1}(\bc)$ with the standard triangular decomposition $\lie b \oplus \lie n^- = \lie n^+ \oplus \lie h \oplus \lie n^-$. We denote the set of roots $R$, 
the positive roots $R^+$, the simple roots $\alpha_i$ for $i=1, \ldots, n$. Then any positive root is of the form
\[\alpha_{i,j} = \alpha_i + \ldots + \alpha_{j}.\]
The set of positive roots of $A_n$ can be arranged in a lower triangular matrix where entries of the $i$-th row are $\alpha_{1,i},...,\alpha_{i,i}$ and entries of the $j$-th column are $\alpha_{j,j}$ to $\alpha_{j,n}$.\\
The set of dominant, integral weights is denoted $P^+$, the set of integral weights $P$ and the fundamental weights $\om_i$ for $i = 1, \ldots, n$. In terms of the dual of the  canonical basis of the diagonal matrices in $M_{n+1}(\bc)$, we can write $\om_i = \epsilon_1 + \ldots + \epsilon_i$ and $\alpha_i = \epsilon_i - \epsilon_{i+1}$.\\
For all $\alpha \in R^+$, we fix a $\lie{sl}_2$-triple $\{ e_\alpha, f_\alpha, h_\alpha:= [e_\alpha, f_\alpha] \}$.


\subsection{Weyl group combinatorics}
Let us denote $W$ the Weyl group, generated by the reflection at simple roots. We can identify $W$ with the symmetric group $S_{n+1}$ via the action on $\epsilon_i$ and hence we write any $w \in W$ as a permutation
\[
w = \left(\begin{array}{cccccc} 1 & 2 & 3 & \ldots & n &n+1\\ w(1) & w(2) & w(3) & \ldots & w(n) & w(n+1) \end{array} \right),
\]
Let us fix $i \in \{1, \ldots, n\}$, then the stabilizer $\operatorname{Stab}_W(\omega_i)$ is isomorphic to $S_{i} \times S_{n+1-i}$, and we denote 
\[W^i = W/\operatorname{Stab}_W(\omega_i).\] 
Each coset in $W/\operatorname{Stab}_W(\omega_i)$ has a unique representative of minimal length and we choose later our favorite reduced decomposition for this representative (Proposition~\ref{prop-fav}).  In abuse of notation we identify each coset with its minimal representative, so we write $w \in W^i$, especially $w_0$ is the unique representative of minimal length in $W^i$ of the longest element from $W$. \\
For $w \in W$ let us denote 
\[
R^-_w = w^{-1}(R^-) \cap R^+.
\]
Then it well-known (and can be easily verified) that $|R^-_w|$ is equal to the length $w$ and
\begin{eqnarray}\label{root-des}
R^-_w = \{  \alpha_{k,j} \in R^+ \mid 1  \leq k < j+1 \leq n+1 \text{ and } w(j+1) < w(k) \}.
\end{eqnarray}
We denote the subalgebra spanned by the root vectors of roots in $-R^-_w $
\[
\lie n^-_w := \langle f_\alpha \, |\, \alpha \in R^-_w \rangle.
\]
Let $w \in W^i$, then we denote for $1 \leq k \leq i$:
\[
\ell_k = \max \{i-1\} \cup \{j \mid i \leq j \leq n \text{ and }  \alpha_{k, j} \in w^{-1}(R^-) \cap R^+\}.
\]
\begin{corollary}\label{coro-ell} Let $1 \leq k \leq i$ and suppose $\ell_k \geq i$.  Then $\alpha_{k, p}  \in w^{-1}(R^-) \cap R^+$ for all $i \leq p \leq \ell_k$. Suppose further $k+1 \leq i$, then $\alpha_{k+1, \ell_k} \in w^{-1}(R^-) \cap R^+$.
\end{corollary}
\begin{proof}
Suppose $i \leq \ell_k$, then $\alpha_{k, \ell_k}  \in w^{-1}(R^-) \cap R^+$, which implies by \eqref{root-des}: $w(\ell_k+1) < w(k)$. Let $k \leq i \leq p < \ell_k$, then we have $w(p+1) < w(\ell_k +1)$, since $w$ is the minimal element modulo $\operatorname{Stab}_W(\omega_i)$. This implies $w(p+1) < w(k)$ and so $\alpha_{k,p} \in R^-_w$.\\
Suppose $1 \leq k <i \leq \ell_k$, then again $w(\ell_k+1) < w(k)$. Further $w(k) < w(k+1)$, since $k+1 \leq i$ and $w$ is the minimal element modulo $\operatorname{Stab}_W(\omega_i)$. This implies $w(\ell_k) < w(k+1)$ and so $\alpha_{k+1, \ell_k} \in w^{-1}(R^-) \cap R^+$.
\end{proof}

This implies that $\ell_k \leq \ell_{k+1}$ for all $k$ and so we obtain a sequence for each $w \in W^i$:
\begin{eqnarray}\label{seq-def}
\underline{\ell}_w := (i-1 \leq \ell_1 \leq \ldots \leq \ell_{i} \leq n).
\end{eqnarray}

\begin{proposition}\label{prop-fav} Let $w \in W^i$, $\underline{\ell}_w$ the corresponding sequence. Then the following is a reduced decomposition of $w$:
\[
w = (s_{\ell_{1} - (i-1)} \cdots s_{1})(s_{\ell_{2} -(i-2)} \cdots s_{2}) \cdots (s_{\ell_{i-1}-1} \cdots s_{i-1}) (s_{\ell_{i}} \cdots s_i)
\]
\end{proposition}
\begin{proof}
Let us denote by $w'$ the right hand side of the equation, then we have to show
\[
w^{-1}(R^+) \cap R^- = w'^{-1}(R^+) \cap R^-.
\]
Since the lengths of both elements are the same, we just have to show that the set of the left hand side is contained in the set on the right hand side. So let $\alpha_{p,q} \in w^{-1}(R^+) \cap R^-$, then by Corollary~\ref{coro-ell}: $i \leq q \leq \ell_p$. Then
\[
(s_{\ell_{p} - (i - p)} \cdots s_{p})\cdots (s_{\ell_{p+1} - (i-p-1)} \cdots s_{p+1}) (s_{\ell_{i}} \cdots s_i)(\alpha_{p,q}) = -(\alpha_{p+q-i}+\alpha_{p+q-i+1}+...+\alpha_{\ell_p-(i-p)}).
\]
$s_{\ell_p - (i-p)}$ does not appear in the remaining part of $w'$, which implies $w' (\alpha_{p,q}) \in R^-$. This implies that $w'$ is a reduced decomposition of $w$, since the number of elements in $w^-(R^+) \cap R^-$ is equal to the number of reflection in the decomposition.
\end{proof}
We see immediately that this gives a one-to-one correspondence between $W^i$ and the set of ascending sequences of $i$ integers in the interval bounded by $i-1$ and $n$.

\subsection{Order and roots}
We have the standard partial order on the set of positive roots, $R^+$, namely $\alpha \geq \beta :\Leftrightarrow \alpha - \beta \in R^+$. By restriction we obtain a partial order on  $R_w^- = w^{-1}(R^-) \cap R^+$. Let $w \in W^i$, then we associate via \eqref{seq-def} a sequence $\underline{\ell}_w$ to $w$. The following proposition is due to the definition of $\mathcal{P}_{\underline{\ell}_w}$ and Corollary~\ref{coro-ell}.
\begin{proposition}\label{prop-pos}
The poset $R^-_w$ is via the map $\alpha_{p,q} \mapsto x_{p,q}$ isomorphic to the poset $\mathcal{P}_{\underline{\ell}_w}$.
\end{proposition}


\section{PBW-graded modules and their bases}\label{Section3}
For $\lambda \in P^+$ we denote the simple finite-dimensional highest weight module of highest weight $\lambda$ by $V(\lambda)$, and a non-zero highest weight vector by $v_\lambda$. Then $V(\lambda) = U(\lie n^-).v_\lambda$ and further $V(\lambda)$ decomposes into $\lie h$-weight space $V(\lambda)_\mu$. For any $w \in W$, the weight space of weight $w(\lambda)$ is one-dimensional and we denote a generator of this line $v_{w(\lambda)}$.


\subsection{Demazure modules} 
\begin{definition}
For $\lambda \in P^+, w \in W$, the Demazure module is defined as
\[
V_w(\lambda) := U(\lie b).v_{w(\lambda)} \subset V(\lambda).
\]
\end{definition}
Note that this is not a $\lie g$-module but a $\lie b$-submodule. For $w = \text{ id }$, this module is nothing but $\bc v_\lambda$ while for $w = w_0$ we have $V_{w_0}(\lambda) = V(\lambda)$. Further, if $w_1, w_2$ are representatives of the same coset in $W^i$, then 
\[
V_{w_1}(m \omega_i) = V_{w_2}(m \omega_i).
\]

\noindent
The Weyl group acts on $U(\lie g)$ as well as on $V(\lambda)$, so we can consider
\[
w^{-1} (V_w(\lambda)) w \subset V(\lambda).
\]
This is equal to 
\[
w^{-1}(U(\lie b).v_{w(\lambda)})w = U(w^{-1} \lie b w).v_\lambda.
\]
Now $w^{-1} \lie b w \subset \lie n^-_w \oplus \lie n^+$. Since $v_\lambda$ is a highest weight vector we have $\lie n^+.v_\lambda = 0$. This implies that
\[
V_w(\lambda) = w(U(\lie n^-_w).v_\lambda) w^{-1}.
\]


\subsection{PBW filtration}
We recall here the PBW filtration. Let $\lie u$ be a finite-dimensional Lie algebra, then we define a filtration on $U(\lie u)$:
\[
U(\lie u)_s := \langle x_{i_1} \cdots x_{i_{\ell}} \, | \, x_{i_j} \in \lie u \, , \, \ell \leq s \rangle.
\]
This induces a filtration on any cyclic $\lie u$-module $M = U(\lie u).m$:
\[
M_s := U(\lie u)_s.m.
\]
We denote the associated graded module $M^a$ (the PBW graded or degenerated module), which is a module for the abelianized version of $\lie u$ 
(the Lie algebra with the same vector space as $\lie u$ but with a trivial Lie bracket), denoted $\lie u^a$ and $U(\lie u^a) = S(\lie u)$, 
the associated graded algebra of $U(\lie u)$.\\
In our case, we consider the algebra $\lie n^-_w \subset \lie n^-$, resp. $\lie b$ and the cyclic modules $U(\lie n^-_w).v_\lambda$, resp. 
$V_w(\lambda)$. We denote the associated graded modules 
\begin{eqnarray}\label{def-modules}
(U(\lie n^-_w).v_\lambda)^a \; \; , \; \;  V_w(\lambda)^a.
\end{eqnarray}
We restrict again ourselves to $\lambda = m \omega_i$. Since $(U(\lie n^-_w).v_\lambda)^a$ is a cyclic $S(\lie n^-_w)$-module, 
there exists an ideal $I_{m,w} \subset S(\lie n^-_w)$ such that
\[
(U(\lie n^-_w).v_\lambda)^a \cong S(\lie n^-_w)/I_{m,w}.
\]
By classical theory we have $f_{\alpha}^{m\omega_i(h_\alpha) + 1} \in I_{m,w}$. 
Let us consider the special case $w_0$ here, then $\lie n^-_w = \lie n^-$ and $V_{w_0}(m \omega_i) = V(m \omega_i)$. 
It has been shown in  \cite{FFoL11a} that in this case the filtered components $V(m \omega_i)_s$ are $U(\lie b)$-modules and especially $U(\lie n^+)$-modules. 
This implies of course that $V(m \omega_i)^a$ is a $U(\lie n^+)$-module. In fact $\lie n^+$ acts on $S(\lie n^-)$ by differential operators $\delta_{\alpha}$ 
(we may omit here scalars):
\[
\delta_\alpha(f_\beta) := e_\alpha.f_\beta = \begin{cases} f_{\beta - \alpha} & \text{ if } \beta - \alpha \in R^+ \\ 0 & \text{ if } \beta-  \alpha \notin R^+ \end{cases}
\]
Then is has been shown in \cite{FFoL11a}:
\begin{theorem*}
$
I_{m,w_0} = S(\lie n^-) \{ U(\lie n^+).f_\alpha^{m \omega_i(h_\alpha) + 1} \, | \, \alpha \in R^+ \}.
$
\end{theorem*}
But this implies that for all $w \in W^i$:
\begin{eqnarray}\label{sub-one}
S(\lie n^-_w) \cap \{ U(\lie n^+).f_\alpha^{m \omega_i(h_\alpha) + 1} \, | \, \alpha \in w^{-1}(R^-) \cap R^+ \} \subset I_{m,w}.
\end{eqnarray}


\subsection{PBW-graded bases}
Let $w \in W$, $\lie n^-_w \subset \lie n^{-}$ the Lie subalgebra associated to $R^-_w$. Let $\bs = (s_\alpha) \in \mathbb{Z}_{\geq 0}^N$. Then we associate to $\bs$ the monomial
\[
f^{\bs} := \prod_{\alpha \in R^-_w} f^{s_\alpha} \in S(\lie n^-_w).
\]
One of the main results of this paper is
\begin{theorem}\label{main-thm} Let $\lambda = m \omega_i \in P^+$, $w \in W^i$, then 
\begin{enumerate}
\item $\{ f^{\mathbf{s}}.v_\lambda \, | \, \mathbf{s} \in  m\mathcal{C}_{\underline{\ell}_w} \cap \mathbb{Z}^{N} \}$ is a basis of $(U(\lie n^-_w).v_\lambda)^a$,
\item $\{ \prod_{\alpha _\in R^-_w} e_{w(\alpha)}^{s_\alpha}.v_m \, | \, \mathbf{s} \in  m\mathcal{C}_{\underline{\ell}_w} \cap \mathbb{Z}^{N}\}$ is a basis for $V_w(\lambda)^a$,
\item$I_{w,m}$ is generated as a $S(\lie n^-_w)$-ideal by $S(\lie n^-_w) \cap \{ U(\lie n^+).f_\alpha^{m \omega_i(h_\alpha) + 1} \, | \, \alpha \in w^{-1}(R^-) \cap R^+ \}.$
\end{enumerate}
\end{theorem}
 The following corollaries are easily deduced from this theorem:
\begin{corollary}
Let $m \omega_i \in P^+$, $w \in W$, then, by fixing an order in each factor, 
\[
\{ \prod_{\alpha \in R^-_w} e_{w(\alpha)}^{s_\alpha}.v_m \, | \,  \mathbf{s} \in  \mathcal{C}^m_{\underline{\ell}_w} \cap \mathbb{Z}^{N}\}
\]
is a basis for $V_w(\lambda)$. Further, the character of the Demazure module is given by
\[
\operatorname{char } V_w(m \omega_i) = e^{w(m \omega_i)} \sum_{ \mathbf{s} \in  m\mathcal{C}_{\underline{\ell}_w} \cap \mathbb{Z}^{N}} e^{w(-\operatorname{wt }(\mathbf{s}))},
\]
where $\operatorname{wt }(\mathbf{s}) := \sum_{\alpha \in R^-_w} s_\alpha \alpha$. 
\end{corollary}

\begin{corollary}\label{coro-2}
Let $w, \tau \in W^i$, and suppose $\tau \leq w$ in the Bruhat order. 
\begin{enumerate}
\item $m\mathcal{C}_{\underline{\ell}_\tau}$ is the face of $m\mathcal{C}_{\underline{\ell}_w}$ defined by setting $s_\alpha = 0$ for all $\alpha \in \left( w^{-1}(R^-) \setminus \tau^{-1}(R^-) \right) \cap R^+$.
\item $ (U(\lie n_{\tau}).v_m)^a \subset (U(\lie n^-_w).v_m)^a.$
\end{enumerate}
\end{corollary}
\begin{proof}
The first follows straight from the definition of $\underline{\ell}_w$ and the second from Theorem~\ref{main-thm} (4) and the first part.
We have certainly by Theorem~\ref{main-thm} (4), $I_{m, \tau} \subset I_{m,w} $ and the rest follows from Corollary~\ref{coro-2}.
\end{proof}


\subsection{Proof for the PBW-graded bases}
To prove the theorem we will follow the ideas presented in \cite{FFoL11a, FFoL11b, FFoL13a, BD14, Fou14b} and show that 
$
\{ f^{\mathbf{s}}.v_\lambda \, | \, \mathbf{s} \in   m\mathcal{C}_{\underline{\ell}_w} \cap \mathbb{Z}^{N} \}
$
spans  $(U(\lie n^-_w).v_\lambda)^a$ (Corollary~\ref{spa-set}) and is linear independent in $V(\lambda)$ (Proposition~\ref{prop-lin}). \\

\begin{proof} We start with proving the spanning property. Since $(U(\lie n^-_w).v_\lambda)^a$ is spanned by applying all monomials in $S(\lie n^-_w)$ to $v_\lambda$, it is clearly enough to prove that if $\mathbf{t}=(t_{\alpha}) \in (\mathbb{Z}_{\geq 0})^{N}$,then 
\[
f^{\mathbf{t}}v_{\lambda} \in \langle f^{\mathbf{s}}v_{\lambda} : \mathbf{s}\in m\mathcal{C}_{\underline{\ell}_w} \cap \mathbb{Z}^{N} \rangle.
\]
For this we introduce a total order $\prec$ on the roots in $R^-_w$ and hence (by Proposition~\ref{prop-pos}) on the poset. We follow here \cite{BD14} and define
\begin{eqnarray}\label{tot-ord}
x_{k_2, j_2}  \prec x_{k_1, j_1} :\Leftrightarrow \left( (j_1 - k_1) > (j_2 - k_2) \right) \text{ or } \left( j_1 - k_1 = j_2 -k_2 \text{ and } k_2 < k_1 \right).
\end{eqnarray}
Note that this extends our partial order $\leq$ and is further a totally ordered subset of the ordered set considered in \cite{BD14}. We consider the induced homogeneous lexicographical order on multisets and hence on monomials in  $S(\lie n^-_w)$.

\begin{lemma}\label{lem-straight}
	Let $\bp$ be a chain in $\mathcal{P}_{\underline{\ell}_w}$ and $\bs \in \mathbb{Z}_{\geq 0}^{N}$ be a multiexponent supported on $\bp$ only.
	Suppose 
	\[
	\sum_{\alpha \in \bp}s_{\alpha} > m,
	\]
	then there exists constants $c_{\bt} \in \bc$, $\bt \in \mathbb{Z}_{\geq 0}^{N}$ such that
	\[
	\left( f^{\bs}+ \sum_{\bt \prec \bs}c_{\bt}f^\bt \right).v_\lambda = 0 \in (U(\lie n^-_w).v_\lambda)^a.
	\]
\end{lemma}
\begin{proof}
We can assume that $\bp$ is a maximal chain in $\mathcal{P}_{\underline{\ell}_w}$, say $\bp =\{ \beta_0,...,\beta_r\}$. Then $\beta_{i} \prec \beta_{i+1}$ and $\beta_{i+1} - \beta_{i} \in R^+$ for $0 \leq i \leq r$. By assumption
\[
| \bs | := \sum_{i=0}^{r}\bs_{\beta_i} > m,
\]
which implies that $f_{\beta_0}^{| \bs |}.v_\lambda = 0  \in (U(\lie n^-_w).v_\lambda)^a$.\\
Set $\gamma_i = \beta_{i-1}-\beta_{i}$ for $1 \leq i \leq r$ and define the operator (following \cite{BD14}):
\[
A=\partial_{\gamma_r}^{\mathbf{s}_{\beta_r}}...\partial_{\gamma_2}^{\mathbf{s}_{\beta_2}+...+\mathbf{s}_{\beta_r}}\partial_{\gamma_1}^{\mathbf{s}_{\beta_1}+...+\mathbf{s}_{\beta_r}}.
\]
This acts certainly as a differential operator of $S(\lie n^-)$. And the key point we will use this that 
\[
A.f_{\beta_0}^{| \bs |} \in S(\lie n^-_w).
\]
This follows since all roots $\alpha_{k,j}$, $1\leq k \leq i , i \leq j \leq n$ with $\alpha_{k,j} \leq  \beta_0$ (here we use the partial order) are in fact in $w^{-1}(R^{-})\cap R^{+}$ Corollary~\ref{coro-ell}. Now following the arguments of \cite{BD14} we see that in $(U(\lie n^-_w).v_\lambda)^a:$
\[
0 = A.f_{\beta_0}^{| \bs | }.v_\lambda = \left( f^{\bs}+ \sum_{\bt \prec \bs}c_{\bt}f^\bt \right).v_\lambda
\]
which proves the lemma.
\end{proof}
Using this straightening law we have immediately the spanning property:
\begin{corollary}\label{spa-set}
Let  $\bt=(t_{\alpha}) \in \mathbb{Z}_{\geq 0}^{N}$, then 
\[f^{\bt}v_{\lambda} \in \langle f^{\bs}v_{\lambda} \, | \, \bs\in m\mathcal{C}_{\underline{\ell}_w} \cap \mathbb{Z}^{N}  \rangle.
\]
\end{corollary}

The proof of Theorem~\ref{main-thm} is complete with Proposition~\ref{prop-lin}:
\end{proof}

\begin{proposition}\label{prop-lin} The set
\[
\{ f^{\mathbf{s}}.v_{\lambda} \, | \,  : \mathbf{s}\in m\mathcal{C}_{\underline{\ell}_w} \cap \mathbb{Z}^{N} \}
\]
is linear independent in $(U(\lie n^-_w).v_\lambda)^a$.
\end{proposition}
For the readers convenience we will give two proofs of this proposition, either using the normality of the polytope or identifying the polytope with a face of a well-studied polytope.

\textit{First proof:}\\
We have for any $\lambda, \mu \in P^+$: $U(\lie n^-_w).v_{\lambda + \mu} \cong U(\lie n^-_w).(v_{\lambda}  \otimes v_{\mu}) \subset V(\lambda) \otimes V(\mu)$. 
It has been shown in \cite{FFoL13} that if $S_w(\lambda)$ (resp. $S_w(\mu)$) parametrizes linear independent subsets of $(U(\lie n^-_w).v_{\lambda})^a$, resp. $(U(\lie n^-_w).v_{\mu})^a$, then $S_w(\lambda) + S_w(\mu)$ parametrizes a linear independent subset in $(U(\lie n^-_w).v_{\lambda\ \otimes v_\mu})^a$. \\
By Lemma~\ref{lem-mink} we know that for all $m,n \geq 1$: 
\[
\left( m\mathcal{C}_{\underline{\ell}_w}  \cap \mathbb{Z}^{N} \right)  + \left(n\mathcal{C}_{\underline{\ell}_w}  \cap \mathbb{Z}^{N} \right)= (m+n)\mathcal{C}_{\underline{\ell}_w} \cap \mathbb{Z}^{N}.
\]
Now, since $\{ f^{\mathbf{s}}v_{\lambda} \, | \,  : \mathbf{s}\in m\mathcal{C}_{\underline{\ell}_w} \cap \mathbb{Z}^{N} \}$ is a spanning set of $(U(\lie n^-_w).v_{m \omega_i})^a$ (Corollary~\ref{spa-set}), it remains to show that this set is a basis for the smallest possible case, namely $m=1$. The proposition follows then by induction.

The lattice points in $\mathcal{C}_{\underline{\ell}_w} $ are nothing but antichains in $R^-_w$ (recall, this poset is isomorphic to $\mathcal{P}_{\underline{\ell}_w}$). The weight of such an antichain is the sum of the corresponding roots. Now it is straightforward to see that different antichains do have different weights. So it remains to show that for a given antichain, $f_{\beta_1} \cdots f_{\beta_s}. v_{m \omega_i} \neq 0 \in U(\lie n^-_w).v_{m \omega_i}$.\\
For this recall that $V(\omega_i) = \bigwedge^i \bc^{n+1}$, and $f_{\alpha_{i,j}}.e_k = \delta_{i,k} e_{j+1}$. Since we are considering an antichain, the roots $\beta_1, \ldots, \beta_s$ are pairwise incomparable, which implies that the root vectors are acting on pairwise distinct elements of the canonical basis and also the images are pairwise distinct elements of the canonical basis. This implies that for an antichain $f_{\beta_1} \cdots f_{\beta_s}. v_{m \omega_i} \neq 0 \in U(\lie n^-_w).v_{m \omega_i}$. On the other hand, the degree of this element is obviously equals to $s$, which implies that this $f_{\beta_1} \cdots f_{\beta_s}. v_{m \omega_i} \neq 0 \in U(\lie n^-_w).v_{m \omega_i}$.\\

\textit{Second proof}:\\
Claim: $\; \; \; m\mathcal{C}_{\underline{\ell}_w}$ is the face of the polytope $m\mathcal{C}_{\underline{\ell}_{w_0}}$ defined by $s_\alpha = 0$ for all $\alpha \notin R^-_w$. \\

\noindent
Let us see first why this claim proves the proposition. $m\mathcal{C}_{\underline{\ell}_{w_0}}$ is the same polytope as $P(m \omega_i)$ defined in \cite{BD14} (in type $A_n$). It was proved there that 
\[
\{ f^{\bs}.v_\lambda \, | \bs \in P(m \omega_i) \} \subset V(m \omega_i)^a
\]
is a basis of $V(m \omega_i)^a$, so especially linear independent. So every subset is linear independent and the proposition follows once we prove the claim.

\textit{Proof of the claim}: The claim follows similarly to \cite[Proposition 2]{Fou14b}: every maximal chain in $\mathcal{P}_{\underline{\ell}_w}$ can be extended to a maximal chain of $\mathcal{P}_{\underline{\ell}_{w_0}}$ and the restriction of every maximal chain from $\mathcal{P}_{\underline{\ell}_{w_0}}$ can be extended to a maximal chain of $\mathcal{P}_{\underline{\ell}_w}$.


\section{Gelfand-Tsetlin polytopes and Kogan faces}\label{Section4}
We will show in this section, that the order polytope can be identified with a Kogan face of a Gelfand-Tsetlin polytope. We then show that the character of this face is the character of the Demazure module corresponding to the chain polytope (of the same poset). \\
Denote $\sigma$ the automorphism of the Dynkin diagram of type $A_n$ and also the induced automorphism of $\lie{sl}_{n+1}$, this induces also an automorphism on the character lattice by $\sigma(\alpha_{i}) = \alpha_{\sigma(i)} = \alpha_{n+1-i}$. For $w \in W$ we have
\begin{eqnarray}\label{dual}
\sigma( \operatorname{char } V_w(\lambda)) = \operatorname{char } V_{\sigma(w)}(\sigma(\lambda)).
\end{eqnarray}
This duality is implicitly used in our construction of Kogan faces. For simplicity, our definition of the character of a Kogan face reflects this automorphism.\\

\noindent
We recall here briefly Gelfand-Tsetlin polytopes (or Cetlin or Zetlin), e.g. the marked order polytopes(see \cite{ABS11}) corresponding to the following poset:\\
The set of vertices is $\{ x_{k,j} \, | \,  0 \leq k \leq n, 1 \leq j \leq n+1-k \}$ and the cover relations are
\[
x_{k-1,j} \geq x_{k,j} \geq x_{k-1,j+1}
\]
for all $k,j$. We do arrange the vertices in a triangle, where the $k$-th row is $\{ x_{k,1}, \ldots, x_{k,n+1-k} \}$. In the example $n = 2$: \\
\begin{picture}(70,70)
\put(260,60){$x_{0,3}$}\put(230,60){$x_{0,2}$}\put(200,60){$x_{0,1}$}\put(215,30){$x_{1,1}$}\put(245,30){$x_{1,2}$}\put(230,0){$x_{2,1}$}
\end{picture}\\

Let $\lambda = \sum m_i \om_i$ and set $s_{0,j} = m_n + \ldots + m_j$. Then the Gelfand-Tsetlin polytope associated to $\lambda$ is defined as
\[
GT(\lambda) = \{ (s_{k,j}) \in \mathbb{R}^{n(n+1)/2} \mid  s_{k-1,j} \geq s_{k,j} \geq s_{k-1,j+1} \}.
\]

Let $\bs \in GT(\lambda)$, then we define the weight of $\bs$ as
\[
\operatorname{wt } \bs := \sum_{k=0}^n \sum_{j=i}^{n+1} (\bs_{k+1,j} - \bs_{k,j}) \epsilon_{k+1}
\]
where we set $\bs_{p,q} = 0$ if $x_{p,q}$ is not a vertex of the Gelfand-Tsetlin poset. Note, that we are dualizing the usual weight of a point in the Gelfand-Tsetlin polytope. Let $S$ be any subset of the lattice points of $GT(\lambda)$, then we define 
\[
\operatorname{char } S := \sum_{\bs \in S} e^{\operatorname{wt } \bs}.
\]


\subsection{Kogan faces}
We consider certain faces of the Gelfand-Tsetlin polytope introduced by Kogan \cite{Kog00}. These faces are given by setting some of the inequalities to equalities. We denote 
\[
A_{i,k} : x_{i,k} = x_{i+1,k}.
\]
Let $F$ be a face defined by some of these inequalities. One can associate a Weyl group element to this face as follows:\\
We start with the identity element and from the bottom row to the top row of the Gelfand-Tsetlin polytope. In each row we go from left to right, for every equality $A_{k,j}$ we multiply $s_{k+j}$ to the right end. The resulting Weyl group element $w(F)$ is called the type of the face.\\

\begin{example}
Let $F$ be the face defined by\\
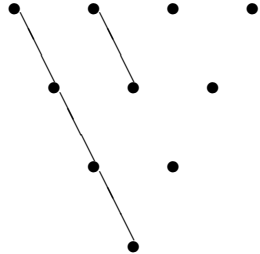
\begin{picture}(0,120)
\put(100,100){$\bullet$} \put(190,100){$\bullet$}\put(160,100){$\bullet$}\put(130,100){$\bullet$}\put(115,70){$\bullet$}\put(145,70){$\bullet$}\put(175,70){$\bullet$}\put(130,40){$\bullet$}\put(160,40){$\bullet$}\put(145,10){$\bullet$}
\put(148,15){$\line(-1,2){13}$}
\put(133,45){$\line(-1,2){13}$}
\put(118,75){$\line(-1,2){13}$}
\put(148,75){$\line(-1,2){13}$}
\end{picture}\\
Then $w(F) = s_3s_2s_1s_2$. 
\end{example}

Note that different faces may have the same type. In the following we will see how such faces are related.\\
Let $F$ be a reduced face such that $A_{j,k}, A_{j, k+1}, A_{j+\ell, k+1 }$ are not equalities of the face for some $j,k,\ell$ but all $A_{j',k'}$ with $j+1 \leq j' \leq j + \ell-1$, $k' \in \{k,k+1\}$ and $A_{j+\ell,k}$ are equalities in $F$. Then we can substitute the equality $A_{j+\ell,k}$ by the equality $A_{j,k+1}$ without changing the type of the face. We call $\ell$ the size of the move.\\
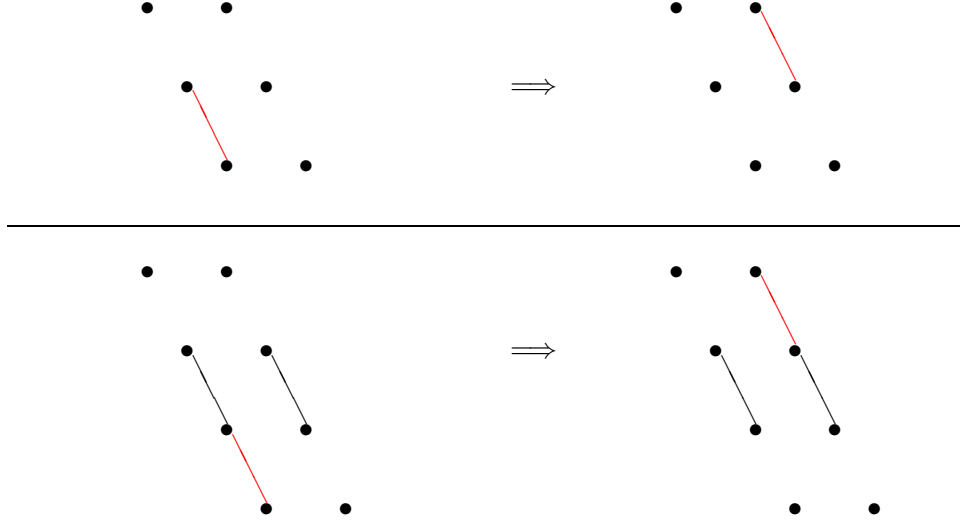
\begin{figure}[h]\label{move-2}
\begin{picture}(500,200)
\put(130,200){$\bullet$}\put(100,200){$\bullet$}\put(115,170){$\bullet$}\put(145,170){$\bullet$}\put(130,140){$\bullet$}\put(133,145){{\color{red}$\line(-1,2){13}$}}\put(160,140){$\bullet$}
\put(240, 170){$\Longrightarrow$}
\put(330,200){$\bullet$}\put(300,200){$\bullet$}\put(315,170){$\bullet$}\put(345,170){$\bullet$}\put(348,175){{\color{red}$\line(-1,2){13}$}}\put(330,140){$\bullet$}\put(360,140){$\bullet$}
\put(50,120){$\line(1,0){360}$}
\put(130,100){$\bullet$}\put(100,100){$\bullet$}\put(115,70){$\bullet$}\put(145,70){$\bullet$}\put(130,40){$\bullet$}\put(133,45){$\line(-1,2){13}$}
\put(160,40){$\bullet$}\put(163,45){$\line(-1,2){13}$}\put(145,10){$\bullet$}\put(148,15){\color{red}{$\line(-1,2){13}$}}\put(175,10){$\bullet$}
\put(240, 70){$\Longrightarrow$}
\put(330,100){$\bullet$}\put(300,100){$\bullet$}\put(315,70){$\bullet$}\put(345,70){$\bullet$}\put(348,75){{\color{red}$\line(-1,2){13}$}}
\put(330,40){$\bullet$}\put(333,45){$\line(-1,2){13}$}\put(360,40){$\bullet$}\put(363,45){$\line(-1,2){13}$}\put(345,10){$\bullet$}\put(375,10){$\bullet$}
\end{picture}
\caption{Ladder moves of size 1 and 2}
\end{figure}
\begin{proposition*}\label{prop-mov}  \cite{Kog00}
Suppose $F, F'$ do have the same type, then there are ladder moves (as above or their inverses) such that $F$ can be transformed into $F'$. There is a unique reduced Kogan face $F_w$ of type $w$ such that all other Kogan faces of type $w$ can be obtained by such ladder moves replacing  $A_{j+\ell,k}$ by  $A_{j,k+1}$. This face is called the Gelfand-Tsetlin face of type $w$.
\end{proposition*}

\begin{corollary}
Let 
\[\tau =    \left(s_{\ell_i +1} \cdots s_{n} \right) \cdots\left( s_{\ell_k - i+k +1} \cdots s_{n- (i-k)}\right) \cdots \left(s_{\ell_1 - i +2} \cdots s_{n- i+1}\right)\]
(in each bracket, the sequence of simple roots is strictly increasing) then the Gelfand-Tsetlin face $F_\tau$ is given by the equalities:
\[
\bigcup_{k = 1}^{i}  \{A_{\ell_k -i+ k , 1}, \ldots , A_{\ell_k - i + k,  n - \ell_k } \}.
\]
\end{corollary}
We see that if $A_{j,k}$ is an equality of the Gelfand-Tsetlin face $F_{\tau}$, then $A_{j,k-1}$ is also an equality of the face:\\
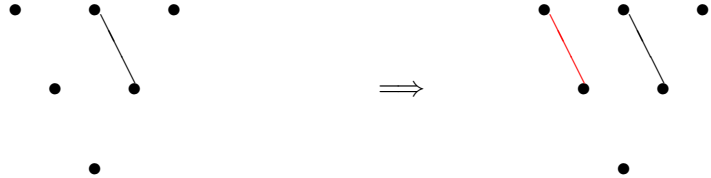
\begin{picture}(70,80)
\put(160,60){$\bullet$} \put(130,60){$\bullet$}\put(100,60){$\bullet$}\put(115,30){$\bullet$}\put(145,30){$\bullet$}\put(148,35){$\line(-1,2){13}$}\put(130,0){$\bullet$}
\put(240, 30){$\Longrightarrow$}
\put(360,60){$\bullet$}\put(330,60){$\bullet$}\put(300,60){$\bullet$}\put(315,30){$\bullet$}\put(318,35){\color{red}{$\line(-1,2){13}$}}\put(345,30){$\bullet$}\put(348,35){$\line(-1,2){13}$}\put(330,0){$\bullet$}
\end{picture}


\subsection{Implicit equations and maximal faces}
In this paper we are considering modules of rectangular highest weight. So if $\lambda = m \omega_i$ and $w \in W$, then we are considering here the Gelfand-Tsetlin polytope $GT(\omega_{n+1-i})$ (due to \eqref{dual}). $\sigma$ induces an bijection $W^i \leftrightarrow W^{n+1-i}$, let $w_0$ then the longest element in $W^i$, then $\sigma(w_0) = w_0^{-1}$ is the longest element in $W^{n+1-i}$. We are interested here in another bijection, namely $w \mapsto ww_0^{-1}$. Then
\[
\left(s_{\ell_1 - (i-1)} \cdots s_1\right)\cdots \left(s_{\ell_i} \cdots s_i\right) \mapsto  \left(s_{\ell_i +1} \cdots s_{n} \right)  \cdots \left(s_{\ell_1 - (i-1) +1} \cdots s_{n- i+1}\right).
\]

Since the highest weight is not regular, we have certain \textit{implicit} equations in each face (in contrast to \textit{explicit} equations defined by fixing the type of the face). In every face we have the following equations due to the fixed highest weight:
\begin{eqnarray}\label{rec-eqn}
\{ A_{k,j} \, | \,   k+j \leq n-i \} \cup \{ A_{k,j} \, | \,  j \geq n-i+1\}
\end{eqnarray}
Now as $x_{0,j} = m$ for $j \leq n-i+1$ and $x_{0,j} = 0$ for $ j  > n-i+1$ we have 
\begin{eqnarray}
x_{k,j} = m \text{ if } k+j \leq n-i+1\, ; \, x_{k,j} = 0 \text{ if } j \geq n-i+2.
\end{eqnarray}
This basically cuts of the upper left and the upper right corner of the Gelfand-Tsetlin triangle. But there are more implicit equations. Suppose $A_{k,j}$ is an (implicit or explicit) equation of a face with $x_{k,j} = m$, then we have the implicit equations: $A_{k+1,j-1}, \ldots, A_{k+j-1,1}$. Further if $A_{k,j}$ is an (explicit or implicit) equation with $x_{k,j} = m$, then all $A_{k-p, j}$ are implicit equations for $p = 0, \ldots, k$.\\
\begin{figure}[H]\label{impl}
\begin{picture}(300,200)
\put(60,190){$\bullet$}\put(30,190){$\bullet$}\put(15,160){$\bullet$}\put(45,160){$\bullet$}\put(75,160){$\bullet$}\put(78,165){$\line(-1,2){13}$}\put(30,130){$\bullet$}\put(60,130){$\bullet$}\put(45,100){$\bullet$}
\put(140, 160){$\Longrightarrow$}
\put(260,190){$\bullet$}\put(230,190){$\bullet$}\put(215,160){$\bullet$}\put(245,160){$\bullet$}\put(275,160){$\bullet$}\put(278,165){$\line(-1,2){13}$}\put(230,130){$\bullet$}\put(260,130){$\bullet$}\put(263,135){{\color{red}$\line(-1,2){13}$}}\put(245,100){$\bullet$}\put(248,105){{\color{red}$\line(-1,2){13}$}}

\put(-50,80){$\line(1,0){360}$}
\put(60,60){$\bullet$}\put(30,60){$\bullet$}\put(0,60){$\bullet$}\put(15,30){$\bullet$}\put(45,30){$\bullet$}\put(30,0){$\bullet$}\put(33,5){$\line(-1,2){13}$}\put(140, 30){$\Longrightarrow$}
\put(260,60){$\bullet$}\put(230,60){$\bullet$}\put(200,60){$\bullet$}\put(215,30){$\bullet$}\put(218,35){\color{red}{$\line(-1,2){13}$}}
\put(245,30){$\bullet$}\put(230,0){$\bullet$}\put(233,5){$\line(-1,2){13}$}
\end{picture}
\caption{Implicit equations}
\end{figure}
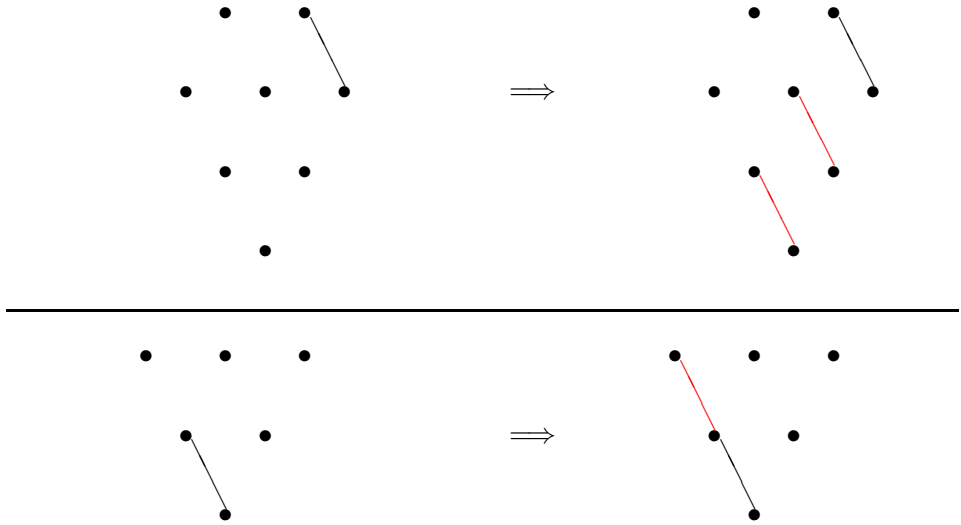

For fixed $0 \leq k \leq n$ we consider the diagonal starting in $x_{k,1}$ and pointing in direction of the upper right corner, $x_{k,1}, x_{k-1, 2}, \ldots...$:\\
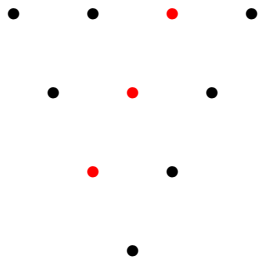
\begin{picture}(70,120)
\put(290,100){$\bullet$}\put(200,100){$\bullet$}\put(260,100){\color{red}{$\bullet$}}
\put(230,100){$\bullet$}\put(215,70){$\bullet$}\put(245,70){\color{red}{$\bullet$}}\put(275,70){$\bullet$}\put(230,40){\color{red}{$\bullet$}}\put(260,40){$\bullet$}\put(245,10){$\bullet$}
\end{picture}\\
Then the first entries in each diagonal are fixed (via implicit or explicit equations) and equal to $m$. We denote the number of such element in the $k$-th column $c_k$ (there might be of course diagonals with $c_k = 0$). Then certainly $c_k = k+1$ for $0 \leq k \leq n-i$. Even more we have for $n-i \leq k \leq n-1$ : $c_k \geq c_{k+1}$ (due to the implicit equations). This implies that the elements $x_{k,j}$ which are fixed by some implicit or explicit equations and equals to $m$ form a staircase pattern above the lower right bound of the Gelfand Tsetlin pattern.\\

\noindent
Let $\tau = ww_0^{-1} =  \left(s_{\ell_i +1} \cdots s_{n} \right) \cdots\left( s_{\ell_k - i+k +1} \cdots s_{n- (i-k)}\right) \cdots \left(s_{\ell_1 - i +2} \cdots s_{n- i+1}\right)$ (in each bracket, the sequence of simple roots is strictly increasing) and suppose $F$ is the Gelfand-Tsetlin face $F_\tau$, then it is easy to see that
\begin{eqnarray}\label{gt-kog}
c^{\tau}_k = \begin{cases} k+1 \text{ if } k \leq n+1-i \\ n - \ell_k \text{ if } k > n+1-i \end{cases}
\end{eqnarray}
and there are no equations (implicit or explicit) for the variables $x_{k,j}$, with 
\[n+1-i\leq k \leq n \; , \; n+1-\ell_k \leq j \leq n+1-i.\]
%

\begin{proposition}\label{move-gt}
Let $F$ be a face of the Gelfand-Tsetlin polytope $GT(m \omega_{n+1-i})$ with $w(F) = \tau$, then $c_{k} \geq c_k^{\tau}$ and there are no further equations (implicit or explicit) for the variables $x_{k,j}$, with $n+1-i \leq k \leq n, c_k +1 \leq j \leq n+1-i$ for all $k$. 
\end{proposition}
\begin{proof} We will use Proposition~\ref{prop-mov} to prove this by induction. The statement is certainly true for $F_\tau$ and suppose the statement is true for $F$. We consider $F'$ obtained from $F$ by a ladder move replacing $A_{j+\ell,k}$ by  $A_{j,k+1}$. There are two cases to be considered:
\begin{itemize}
\item The ladder move is of size $\ell > 1$ (Figure~\ref{move-2}). So we have $A_{j,k}, A_{j, k+1}, A_{j+\ell, k+1}$ are not explicit equations of $F$ but all $A_{j',k'}$ with $j+1 \leq j' \leq j + \ell-1$, $k' \in \{k,k+1\}$ and $A_{j+\ell,k}$ are equalities in $F$. Let $F'$ be obtained by replacing $A_{j + \ell, k}$ by $A_{j,k+1}$. \\
By assumption, for $j' = j + \ell-1, k' = k+1$, we have that $A_{j+\ell-1, k+ 1}$ is an explicit equation of the faces $F$ and $F'$.  Then we have, see Figure~\ref{impl}, $A_{j + \ell, k}$ as an implicit face of $F'$. Further, since by assumption $A_{j+1, k+1}$ is an explicit face of $F$ and $F'$, and since by induction the equalities of $F$ form the described pattern, we have $A_{j,k+1}$ is an implicit equation in $F$. This implies that $F = F'$.

\item The ladder move is of size $\ell = 1$ (Figure~\ref{move-2}). Then $A_{j,k}, A_{j, k+1}, A_{j+1, k+1}$ are not explicit equations of $F$.  Suppose we have fixed $x_{j,k+1} = m$ in $F$ by an implicit or explicit equation. Then $A_{j+1,k}$ is an implicit equation in $F'$ and hence $F' \subset  F$ (they may differ only by the equation $A_{j,k+1}$). \\
Suppose $x_{j,k+1}$ is not fixed by an implicit or explicit equation in $F$. Then by induction $A_{j-1,k+1}$ is not an implicit or explicit equation of $F$. But then $A_{j,k}$ is not an implicit equation of $F$, while $x_{j+1,k} = m$ by assumption. So $A_{j,k}$ has to be an explicit equation which is a contradiction (since we are considering ladder moves here and therefore $A_{j,k}$ is not an explicit equation).
\end{itemize}
\end{proof}
The set of Kogan faces of type $\tau$ is partially ordered by inclusion. From the proof of Proposition~\ref{move-gt} we have immediately the following lemma:
\begin{corollary}\label{kog-ord}
Ladder moves on Kogan faces are compatible with this partial order. So if $F, F'$ are Kogan faces of type $\tau$ and $F'$ is constructed through a ladder move from $F$, then $F' \subset F$. Especially, $F \subset F_\tau$  the Gelfand-Tsetlin face of type $\tau$. 
\end{corollary}

The goal of the section was to prove the following lemma which follows immediately from Corollary~\ref{kog-ord}:
\begin{lemma}
Let $w = (s_{\ell_1 - (i-1)} \cdots s_1)\cdots( s_{\ell_i} \cdots s_i)$, then there is a unique maximal Kogan face  $F_{ww_0^{-1}}(\omega_{n+1-i}) \subset GT(\omega_{n+1-i})$ of type $ww_0^{-1}$.
\end{lemma}

\begin{corollary}\label{pol-equal}
Let $w \in W^i$, then $F_{ww_0^{-1}}(\omega_{n+1-i})$ is isomorphic to the order polytope $\mathcal{O}_{\underline{\ell}_w} $. Further $F_{ww_0^{-1}}(m\omega_{n+1-i})$ is isomorphic to the $m$-th dilation of the order polytope.
\end{corollary}
\begin{proof}
The proof is straightforward. We see that by \eqref{gt-kog} the \textit{free} variables in $F_{ww_0^{-1}}(\omega_{n+1-i})$ form a pattern inside a rectangle of width $i$ and height $n-i$, and the heights of the columns are $\ell_k - (k-1)$. So the pattern is the same as the pattern of the poset $\mathcal{P}_{\underline{\ell}_w}$ (see \eqref{gt-kog}). The inequalities defining the Gelfand-Tsetlin polytope are obviously the cover relation of the poset $\mathcal{P}_{\underline{\ell}_w}$. Which implies that $F_{ww_0^{-1}}(\omega_{n+1-i})$ is isomorphic to the order polytope $\mathcal{O}_{\underline{\ell}_w}$. The statement about the $m$-dilation follows straight from this.
\end{proof}


\subsection{Kogan face and Demazure character}
We finish the section as promised, showing that the character of the maximal Kogan face is the character of a corresponding Demazure module.

\begin{lemma}
Fix $w \in W^i$ and consider the maximal Kogan face $F_{ww_0^{-1}}(m\omega_{n+1-i}) \subset GT(m\omega_{n+1-i})$, then
\[
\operatorname{char } V_w(m\omega_i) = \operatorname{char } F_{ww_0^{-1}}(m\omega_{n+1-i}).
\]
\end{lemma}
\begin{proof}
Let us recall the well-known identification of lattice points in the Gelfand-Tsetlin polytope with semi--standard Young Tableaux (see for example \cite{DLM04}). Due to our definition of the weight of a Gelfand-Tsetlin pattern, we identify the lattice points in $GT(m\omega_{n+1-i})$ with semi--standard Young Tableaux of shape $m\omega_i$ (e.g. $m$ columns of height $i$).\\
Let $\bs \in GT(m\omega_{n+1-i})$, then we associate the following product of Kashiwara operators
\[
A_{\bs} := \left(f_{n-(i-1)}^{m - \bs_{n+1-i,1}} \cdots f_1^{m - \bs_{1,n+1-i}}\right)\cdots\left(f_{n-k}^{m-\bs_{n-k,1}} \cdots f_{i-k}^{m-\bs_{i-k,n+1-i}}\right)\cdots\left (f_n^{m-\bs_{n,1}} \cdots f_i^{m-\bs_{i,n+1-i}}\right).
\]
The map $\bs \mapsto A_{\bs}.b_{m \omega_i}$, where $b_{m \omega_i}$ is the Kashiwara-Nakashima tableaux of weight $m \omega_i$, is the well-known weight preserving bijection between lattice points in $GT(m \omega_{n+1-i})$ and semi--standard Young Tableaux of shape $m \omega_i$.\\
Let $\bs \in F_{ww_0^{-1}}(m \omega_{n+1-i})$, and consider the $k$-th factor of $A_{\bs}$, $A_{\bs,k} = f_{n-k}^{m-\bs_{n-k,1}} \cdots f_{i-k}^{m-\bs_{i-k,n+1-i}}$. We have seen that we have certain implicit and explicit equations in the face, setting some of the $\bs_{k,j} = m$. Then by \eqref{gt-kog}: 
\begin{eqnarray}\label{red-A}
A_{\bs,k} = f_{\ell_{i-k} -k}^{m-\bs_{\ell_{i-k} -k,1}} \cdots f_{i-k}^{m-\bs_{i-k,n+1-i}}.
\end{eqnarray}
Recall the revised Demazure character formula (see \cite{Lit95, Kas93}), for given $\tau = s_{i_1} \cdots s_{i_r}$ and $\lambda \in P^+$:
\begin{eqnarray}\label{rev-D}
\operatorname{char } V_\tau(\lambda) = \sum_{b \in T_\tau(\lambda)} e^{\operatorname{wt } b}
\end{eqnarray}
where 
\[
T_\tau(\lambda) = \{ f_{i_1}^{t_1} \cdots f_{i_s}^{t_s}.b_\lambda \mid t_j \geq 0 \}.
\]
So the character of the Demazure module is equal to the character of the Demazure crystal $T_\tau(\lambda)$. \\
Recall that $w = (s_{\ell_1 - (i-1)} \cdots s_1)\cdots( s_{\ell_i} \cdots s_i)$ and so we see from \eqref{red-A} immediately that for any lattice point $\bs \in F_{ww_0^{-1}}(m \omega_{n+1-i})$:
\[
A_{\bs}.b_{m \omega_i} \in T_w(m\omega_i).
\] 
Now, by Corollary~\ref{pol-equal}, we have $F_{ww_0^{-1}}(m \omega_{n+1-i}) = m\mathcal{O}_{\underline{\ell}_w}$, and especially the number of lattice points is equal. On the other hand, the number of lattice points in $m\mathcal{O}_{\underline{\ell}_w}$ is equal to the number of lattice points in $\mathcal{C}_{\underline{\ell}_w} $ (Theorem~\ref{sta-1}). Again, this is equal to the dimension of $V_w(\omega_i)$ (Theorem~\ref{main-thm}), which is nothing but $|T_w(m\omega_i)|$ (by \eqref{rev-D}). But this implies that
\[
\{ A_{\bs}.b_{m \omega_i} \mid \bs \in F_{ww_0^{-1}}(m \omega_{n+1-i}) \} = T_w(m\omega_i),
\]
and this finishes the proof.
\end{proof}

\begin{remark}
This has been shown for regular $\lambda$ in \cite{KST12}, using an interpretation of Demazure operators in terms of Kogan faces. We are convinced that their methods do apply also in the case of rectangular weights.
\end{remark}


\section{Degenerations}\label{Section5}
\subsection{Demazure modules are favourable modules}
In \cite{FFoL13a}, the notion of favourable modules has been introduced, let us recall this here briefly and show that the Demazure modules $V_w(m \omega_i)$ are favourable $\lie n^+$-modules.\\
We denote the group of upper triangular matrices with determinant equals to $1$ by $\mathbb{U}$, then $\lie n^+$ is the corresponding Lie algebra.  Let $M$ an $\mathbb{U}$-module such that there exists $m \in M$ with $M = U(\lie n^+).m$. Let $\{x_1, \ldots x_N\}$ be an ordered basis of $\lie n^+$ and we fix an induced homogeneous lexicographic order on monomials in $U(\lie n^+)$. This induces a filtration on $M$
\[
M^{\bp} := \langle f^{\bq}.m \, | \, \bq \preceq \bp \rangle_\bc
\]
where in the associated graded module any graded component is at most one-dimensional. We say $\bp$ is essential if $M^{\bp}/M^{ < \bp}$ is non-zero and denote the set of essential monomials $es (M)$. We can view this as a subset in $\mathbb{Z}_{\geq 0}^N$.\\
$M$ is called a favourable $\mathbb{U}$-module if there exists a convex polytope $P(M) \subset \mathbb{Z}^N$ whose set of lattice points $S(M)$ is equal to $es(M)$ and 
\[
|n S(M) | = \operatorname{dim } U(\lie n^+). m^{\otimes n} \subset M^{\otimes n}.
\]
\begin{lemma}\label{lem-fav}
Let $w \in W$, $\lambda = m \omega_i$, then $V_w(m \omega_i)$ is a favourable $\mathbb{U}$-module via the polytope $\mathcal{C}_{\underline{\ell}_w}$.
\end{lemma}
\begin{proof}
Recall the total order $\preceq$ on $R^-_w$ \eqref{tot-ord}. By applying $w$ to the roots in $-R^-_w$ and extending then arbitrary to all other positive roots we obtain a total order on $R^+$. Then we have seen Theorem~~\ref{main-thm} that
\[
\{ \prod_{\alpha _\in R^-_w} e_{w(\alpha)}^{s_\alpha}.v_\lambda \, | \, \mathbf{s} \in  \mathcal{C}^m_{\underline{\ell}_w} \cap \mathbb{Z}^{N}\}
\]
is a basis of the associated graded space $V_w(m \omega_i)^a$ and by the straightening law Lemma~\ref{lem-straight}, we see that this is actually a basis of $V_w(m \omega_i)^t$ and so the lattice points in $\mathcal{C}_{\underline{\ell_w}}$ are the essential monomials. Since $\mathcal{C}_{\underline{\ell_w}}$ is normal and 
\[
V_w((m+n) \omega_i) \cong U(\lie n^+).v_{w(m \omega_i)} \otimes v_{w(n \omega_i)} \subset V_w(m \omega_i) \otimes V_w(n \omega_i)
\]
we see that also the second condition for a favourable module is satisfied. 
\end{proof}

Recall, that $M^a$ denotes the associated PBW-graded module of $M$, and we denote $M^t$ the associated graded module of $M$ with respect to the filtration induced by $\preceq$. Then $M^a, M^t$ are both modules for $\mathbb{U}^a$, which is the algebraic group of $N$-copies of the abelian group, as well as cyclic modules for $S(\lie n^+)$, the symmetric algebra of the vector space $\lie n^+$. We denote 
\[
X_w := \overline{\mathbb{U}.[v_{m \omega_i}]} \subset \mathbb{P}(V_w(m\omega_i)) \subset \mathbb{P}(V(m \omega_i)) = \operatorname{Gr}(i,n+1)
\]
the minuscule Schubert variety associated to $\omega_i$ and $w$, and the degenerated Schubert varieties
\[
X_w^a := \overline{\mathbb{U}^a.[v_{m \omega_i}]} \subset \mathbb{P}(V_w(m\omega_i)^a)  \; ; \; X_w^t := \overline{\mathbb{U}^a.[v_{m \omega_i}]} \subset \mathbb{P}(V_w(m\omega_i)^t).
\]

With Lemma~\ref{lem-fav} we can use the main theorem in \cite[Main Theorem]{FFoL13a}
 to deduce
\begin{theorem} Let $w \in W$, $\lambda = m \omega_i$, then
\begin{enumerate}
\item $X_w^t$ is a toric variety.
\item $X_w^t$ is a flat degeneration of $X_w^a$ and both are flat degenerations of $X_w$.
\item $X_w^t, X_w^a$ are both projectively normal and arithmetically Cohen-Macaulay varieties.
\item The polytope $\mathcal{C}_{\underline{\ell}_w}$ is the Newton-Okounkov body for $X_w$ and its abelianized version.
\end{enumerate}
\end{theorem}


\subsection{Comparison with Gelfand-Tsetlin degenerations}\label{compare-pol}
Let $\lambda = m \omega_i$, $w \in W^i$, $\underline{\ell}_{w} = (\ell_1 \leq \ldots \leq \ell_i)$ be the corresponding sequence \eqref{seq-def}. We want to compare the toric degenerations obtained through the Gelfand-Tsetlin polytope (here $\mathcal{O}_{\underline{\ell}_w}$) and our polytope obtained via the PBW-grading ($\mathcal{C}_{\underline{\ell}_w}$). We can apply Theorem~\ref{thm-equiv} to $\mathcal{P}_{\underline{\ell}_w}$.
\begin{lemma}\label{lem-compare} Let $w \in W^i$ and $\underline{\ell}_w$ the associated sequence. The toric variety obtained via the Gelfand-Tsetlin polytope is isomorphic to $X_w^t$ if and only if $\ell_{i-2}  <  i+1$ or $\ell_{i-1} < i+2$. 
\end{lemma}

We can reformulate the conditions of the lemma. The normal fans are non-isomorphic if and only if there is a reduced decomposition of $w$ of the form
\[
w = \ldots (s_{i-1} s_{i-2})(s_{i+1}s_i s_{i-1})(s_{i+2}s_{i+1}s_i).
\]

\begin{remark}
It would be interesting to see how the degenerations studied in \cite{Lit98, AB04}, induced from a reduced decomposition of the longest Weyl group element, are related to our degenerations. So is the degeneration through the PBW filtration actually a new one or isomorphic to a previous known one. Nevertheless, one advantage of our construction is that a lot of data are given very explicit, like the facet, the vertices etc. 
\end{remark}

\subsection{Gorenstein polytopes}
For the sake of completion let us finish with a brief view towards Gorenstein polytopes.\\
Let $w \in W^i$ and suppose $\mathcal{P}_{\underline{\ell}_{_w}}$ is a pure poset (or ranked poset), e.g. all maximal chains have the same length. Then a result by Stanley \cite[Theorem 5.4]{Sta78} gives that the order polytope $\mathcal{O}_{\underline{\ell}_{w}}$ is a Gorenstein polytope, e.g. a integer dilation of the polytope is (up to translation by a vector) reflexive. 
This is equivalent to the fact that the $h^*$-vector of $\mathcal{O}_{\underline{\ell}_{w}}$ is symmetric (\cite{Sta78}).\\ 
Let us translate this condition to our context:\\
Denote $\{\ell_{j_1}, \ldots, \ell_{j_s} \} \subset \{ \ell_1 \leq \ldots \leq \ell_i\}$ the subset such that $\ell_{j_p} \neq \ell_{j_p -1}$. Then all maximal chains have the same length if and only if $\ell_{j_p} - j_p$ is the same integer for all $p=1, \ldots, s$.\\
Since $\mathcal{O}_{\underline{\ell}_{w}}$ and the chain polytope have the same Ehrhart polynomial, they also have the same $h^*$-vector.
\begin{corollary}\label{cor-gor}
Let $w \in W^i$, then the order polytope $\mathcal{O}_{\underline{\ell}_{w}}$ and the chain polytope $\mathcal{C}_{\underline{\ell}_{w}}$ are Gorenstein polytopes if and only if there exists $K$ such that for all $j = 1, \ldots, i$: 
\[ 
\ell_j - j = K \text{ or } \ell_j = \ell_{j-1}.
\]
\end{corollary}

\bibliographystyle{alpha}
\bibliography{biswal-fourier}
\end{document}